\documentclass[a4paper,12pt]{article}
\usepackage{hyperref}
\usepackage[T1]{fontenc}
\usepackage[utf8]{inputenc}

\usepackage[english]{babel}
\usepackage{mathtools}
\usepackage{amsfonts}
\usepackage{color}
\usepackage{amsmath}
\usepackage{amsthm}
\usepackage{amssymb}
\usepackage{mathrsfs}
\usepackage{amstext}

\usepackage{authblk}
\usepackage{stmaryrd}
\usepackage{makeidx}
\makeindex

\numberwithin{equation}{section}

\theoremstyle{definition}
\newtheorem{thm}{Theorem}[section]
\newtheorem{rem}[thm]{Remark}

\newtheorem{lem}[thm]{Lemma}
\newtheorem{cor}[thm]{Corollary}

\newtheorem{conj}[thm]{Conjecture}
\newtheorem{q}[thm]{Question}

\newtheorem*{thm*}{Theorem}
\newtheorem*{rem*}{Remark}
\newtheorem*{folg*}{Folgerung}
\newtheorem*{examples*}{Beispiele}
\newtheorem*{ex*}{Beispiel}
\newtheorem*{lem*}{Lemma}
\newtheorem*{prop*}{Proposition}
\newtheorem*{defi*}{Definition}
\newtheorem*{exercise*}{Übung}
\newtheorem*{conj*}{Conjecture}
\newtheorem*{q*}{Question}

\usepackage{cleveref}

\newcommand{\bP}{\mathbf{P}}

\newcommand{\bbN}{\mathbb{N}}

\newcommand{\bbR}{\mathbb{R}}

\newcommand{\cE}{\mathcal{E}}
\newcommand{\cF}{\mathcal{F}}
\newcommand{\cG}{\mathcal{G}}

\newcommand{\cP}{\mathcal{P}}

\newcommand{\cT}{\mathcal{T}}
\newcommand{\cU}{\mathcal{U}}

\newcommand{\es}{\emptyset}

\newcommand{\Lra}{\Leftrightarrow}
\newcommand{\ra}{\rightarrow}

\newcommand{\Ra}{\Rightarrow}
\newcommand{\sm}{\setminus}
\newcommand{\sbse}{\subseteq}

\newcommand{\spse}{\supseteq}

\newcommand{\Fix}{\operatorname{Fix}}

\author{Nicolas Nagel}
\affil{Department of Mathematics, TU Chemnitz, Germany \\ nicolas.nagel@math.tu-chemnitz.de}
\title{Notes on the Union Closed Sets Conjecture}
\date{}

\begin{document}

\maketitle


\begin{abstract}
The \emph{Union Closed Sets Conjecture} states that in every finite, nontrivial set family closed under taking unions there is an element contained in at least half of all the sets of the family. We investigate two new directions with respect to the conjecture. Firstly, we consider the frequencies of all elements among a union closed family and pose a question generalizing the Union Closed Sets Conjecture. Secondly, we investigate structures equivalent to union closed families and obtain a weakening of the Union Closed Sets Conjecture. We pose some new open questions about union closed families and related structures and hint at some further directions of research regarding the conjecture.
\end{abstract}

\section{Introduction}

For an integer $n \in \bbN$ set $[n] \coloneqq \{1, \dots, n\}$ and let $\cP(n) \coloneqq \cP([n])$ be the power set over $[n]$. A family $\cF \sbse \cP(n)$ is called \emph{nontrivial} if
$$
\bigcup \cF \coloneqq \bigcup_{F \in \cF} F = [n]
$$
($\bigcap \cF$ being defined similarly) and \emph{union closed} if for $A, B \in \cF$ also $A \cup B \in \cF$. Notice that for any nontrivial, union closed family $\cF \sbse \cP(n)$ we have $[n] = \bigcup \cF \in \cF$.

\begin{conj} [Union Closed Sets Conjecture] \label{conj:UCS}
For every nontrivial, union closed $\cF \sbse \cP(n)$ there is an $x \in [n]$ with
$$
\#\{F \in \cF: x \in F\} \geq \frac{1}{2} \cdot \#\cF.
$$
\end{conj}

Here $\# X$ denotes the cardinality of a set $X$. If we desire, we may assume that $\es \in \cF$, since this only makes the above conjecture harder. Dating back to at least the 1980s, the Union Closed Sets Conjecture \ref{conj:UCS} (also referred to as \emph{Frankl's conjecture}) has a long and rich history. For more details on the development of the conjecture and what is known about it see the survey \cite{BS15}. In recent years \cite{AEL20, ADPRT20, BD18, EIL22, Kar17, Mas16, Raz17, Stu20, Tia21, VZ17} have further been published investigating the conjecture with respect to one aspect or another. In this context, the collaborative effort in \cite{Gow16} should also be mentioned.

Even though the Union Closed Sets Conjecture \ref{conj:UCS} has a rather simple statement, so far no proof or counterexample is known. In this paper, we want to investigate the conjecture in two new ways. Firstly, we may reformulate the conjecture to the statement that for every nontrivial, union closed family $\cF$ the \emph{most frequent} element among $\cF$ is contained in at least half of all sets of $\cF$. It is natural to ask, whether we can make similar statements about the other less frequent elements, which leads to a generalization of the conjecture.

Secondly, we will investigate structures equivalent (\emph{cryptomorphic}) to union closed families. While studying the Union Closed Sets Conjecture \ref{conj:UCS} by use of equivalent structures is not new, it seems that this approach has not been used to its fullest yet. These equivalent structures enable us to get a better idea of the inner structure of union closed families. In particular, we will prove a weaker version of the conjecture. These ideas, while not being strong enough to prove the conjecture itself, certainly give new insights on how one can study union closed families further. Again, we pose some new questions that give a hint on how to approach the conjecture in a new way.

\section{Frequencies}

Let $\cF \sbse \cP(n)$ be a nontrivial, union closed family. By permuting the elements of the \emph{ground set} $[n]$, we may assume
\begin{align} \label{ineq:Ordering}
\#\{F \in \cF: 1 \in F\} \geq \#\{F \in \cF: 2 \in F\} \geq \dots \geq \#\{F \in \cF: n \in F\}.
\end{align}
The Union Closed Sets Conjecture \ref{conj:UCS} is then equivalent to the statement that
$$
\#\{F \in \cF: 1 \in F\} \geq \frac{1}{2} \cdot \#\cF.
$$
What can be said about the other, less frequent elements? Notice that, assuming the Union Closed Sets Conjecture \ref{conj:UCS} holds, the family
$$
\cF' \coloneqq \{F \sm \{1\}: 1 \in F \in \cF\}
$$
is again a nontrivial, union closed family of size at least $\frac{1}{2} \cdot \#\cF$ now over the ground set $\{2, \dots, n\}$. Applying the Union Closed Sets Conjecture \ref{conj:UCS} to $\cF'$ we get an element $x \in \{2, \dots, n\}$ with
$$
\#\{F \in \cF: x \in F\} \geq \#\{F \in \cF': x \in F\} \geq \frac{1}{2} \cdot \#\cF' \geq \frac{1}{4} \cdot \#\cF.
$$
In particular, since the element $2$ is at least as frequent in $\cF$ as $x$ we get
$$
\#\{F \in \cF: 2 \in F\} \geq \frac{1}{4} \cdot \#\cF.
$$
Iterating this argument we have shown (assuming the Union Closed Sets Conjecture \ref{conj:UCS})
\begin{align} \label{ineq:GeneralizedUCS}
\#\{F \in \cF: k \in F\} \geq \frac{1}{2^k} \cdot \#\cF
\end{align}
for all $k \in [n]$. However, one might feel that every iterative step is a bit wasteful since we only consider sets from $\cF'$ even though there might be a lot of sets in $\cF \sm \cF'$ also containing $2$ or any other element. We therefore ask whether the right hand side of \eqref{ineq:GeneralizedUCS} can be improved. We suggest that this is indeed the case and that even the constant fraction on the right hand side can be raised.

\begin{q} \label{que:UCSFrequencies}
Let $\cF \sbse \cP(n)$ be a nontrvial, union closed family fulfilling \eqref{ineq:Ordering}, is it then true that
\begin{align} \label{ineq:Frequencies}
\#\{F \in \cF: k \in F\} \geq \frac{1}{2^{k-1}+1} \cdot \#\cF
\end{align}
for all $k \in [n]$?
\end{q}

Notice that $k=1$ yields the Union Closed Sets Conjecture \ref{conj:UCS}. If Question \ref{que:UCSFrequencies} indeed holds then the constant $1/(2^{k-1}+1)$ on the right hand side of \eqref{ineq:Frequencies} is optimal. For this fix $k \in [n]$ and consider the family
\begin{align} \label{eq:ExampleEquality}
\cF = \cP(k-1) \cup \{[n]\},
\end{align}
which already fulfils \eqref{ineq:Ordering}. Here $\#\cF = 2^{k-1}+1$ and $\{F \in \cF: k \in F\} = \{[n]\}$, so that \eqref{ineq:Frequencies} holds with equality. While this shows that the constant for this specific $k$ cannot be improved (if it holds at all), notice how \eqref{ineq:Frequencies} is very far off for all the other $l \in [n] \sm \{k\}$. This leads to a multitude of other related questions one could ask in this context.

\begin{q}
How do the families $\cF$ look where equality is achieved in \eqref{ineq:Frequencies} for some $k \in [n]$? Are there other extremal construction next to \eqref{eq:ExampleEquality}?
\end{q}

A family $\cF \sbse \cP(n)$ is called \emph{separating} if
\begin{align} \label{eq:Separating}
\{F \in \cF: x \in F\} \neq \{F \in \cF: y \in F\}
\end{align}
for all $x, y \in [n], x \neq y$. Notice how in general (for $k < n$) the construction \eqref{eq:ExampleEquality} is not separating.

\begin{q}
Can the right hand side of \eqref{ineq:Frequencies} be improved if in addition we assume $\cF$ to be separating?
\end{q}

%

We will now give some justification for Question \ref{que:UCSFrequencies}. The general idea would be an inductive proof of \eqref{ineq:Frequencies} via induction on $k$ with the induction start $k=n$. Our aim should thus be to prove Question \ref{que:UCSFrequencies} for large $k$ first. To do this, we use the following lemma.

\begin{lem} \label{lem:Frequencies}
Let $\cF \sbse \cP(n)$ be a nontrivial, union closed family and $x \in [n]$. For every $A \in \cF$ with $x \in A$ it holds
$$
\{F \in \cF: x \in F\} \geq \frac{1}{2^{\#A-1}+1} \cdot \#\cF.
$$
\end{lem}

\begin{rem}
To get large lower bounds on the frequency of $x$ we would wish to choose a small set $A \in \cF$ containing $x$. The proof of the lemma will be analogous to the proof of the ``folklore'' theorem (see \cite{BS15}) that if $\cF$ is a nontrivial, union closed family with $\{x\} \in \cF$ (that is it contains a singleton) then $x$ is contained in at least half of all the sets of $\cF$.
\end{rem}

\begin{proof} [Proof of Lemma \ref{lem:Frequencies}]
Consider the surjective map
$$
\varphi: \{X \sbse [n]: x \notin X\} \ra \{Y \sbse [n]: A \sbse Y\}, X \mapsto X \cup A.
$$
We claim that for every $Y \sbse [n], A \sbse Y$ the fiber $\varphi^{-1}(Y)$ is of cardinality $2^{\# A -1}$. Indeed, it even holds that
$$
\varphi^{-1}(Y) = \{X \sbse [n]\sm\{x\}: X \cap ([n] \sm A) = Y \cap ([n] \sm A)\},
$$
which is clear by the elementary equivalence
$$
X \cup A = Y \quad \Lra \quad X \sm A = Y \sm A.
$$
Since $Y \cap ([n] \sm A)$ is a fixed set, we can only vary $X$ on $A \sm \{x\}$, so that
$$
\#\varphi^{-1}(Y) = 2^{\#(A \sm \{x\})} = 2^{\#A-1}.
$$
Let now $\cF$, $x$ and $A$ be as in the statement. Assume that $\{G \in \cF: x \notin G\} \neq \es$, otherwise the claim of the lemma is clear. Notice then, since $\cF$ is union closed, the above map $\varphi$ restricts to a map
$$
\varphi: \{G \in \cF: x \notin G\} \ra \{F \in \cF: x \in F\}, G \mapsto G \cup A.
$$
By the first part of the proof, for every $F \in \cF, x \in F$ there are at most $2^{\#A-1}$ many $G \in \cF, x \notin G$ with $G \cup A = F$. Thus
$$
\#\{G \in \cF: x \notin G\} \leq 2^{\#A-1} \cdot \#\{F \in \cF: x \in F\}
$$
and using $\#\{G \in \cF: x \notin G\} = \#\cF -\#\{F \in \cF: x \in F\}$ we proved the lemma.
\end{proof}

\begin{thm} \label{thm:Frequencies}
Question \ref{que:UCSFrequencies} holds for $k = n$ and $k = n-1$.
\end{thm}


\begin{proof} 
The case $k=n$ follows by applying Lemma \ref{lem:Frequencies} to $x = n$ and $A = [n]$. For $k = n-1$, we may assume $\cF$ to be separating (see \eqref{eq:Separating}). Otherwise, we may combine elements that cannot be separated and we get a separating, nontrivial, union closed family over a smaller ground set on which we could work instead. Since then
$$
\{F \in \cF: n-1 \in F\} \neq \{F \in \cF: n \in F\}
$$
but also
$$
\#\{F \in \cF: n-1 \in F\} \geq \#\{F \in \cF: n \in F\},
$$
there must be an
$$
A \in \{F \in \cF: n-1 \in F\} \sm \{F \in \cF: n \in F\}.
$$
Notice that $A \neq [n]$, so in particular $\#A \leq n-1$. Apply now Lemma \ref{lem:Frequencies} with $x = n-1$ and this $A$ to get the bound for $k=n-1$.
\end{proof}

To prove the statement of Question \ref{que:UCSFrequencies} it remains to proceed by induction, that is if the statement holds for $k > 1$ we would need to show it for $k-1$. It remains open how this could be done precisely.


\section{Weakenings of the Union Closed Sets Conjecture}

We now present some different ideas for the Union Closed Sets Conjecture \ref{conj:UCS}. We begin by introducing some general theory which we will later apply to study the structure of union closed families.

\subsection{Equivalent Structures}

To start, we recapitulate some of the already applied structures to investigate the Union Closed Sets Conjecture \ref{conj:UCS} (see \cite{BS15}). It is easy to show that a family $\cF \sbse \cP(n)$ is union closed if and only if $\cG \coloneqq \{[n] \sm F: F \in \cF\}$ is \emph{intersection closed} (that is if $A, B \in \cG$ also $A \cap B \in \cG$). Furthermore, an element $x$ is contained in at least half of all sets of $\cF$ if and only if it is contained in at most half of all sets of $\cG$. This observation gives the \emph{Intersection Closed Sets Conjecture}:

For every intersection closed family $\cG \sbse \cP(n)$ with $\bigcap \cG = \es$ (the nontriviality condition) there is an $x \in [n]$ with
$$
\#\{G \in \cG: x \in G\} \leq \frac{1}{2} \cdot \#\cG.
$$

Of course, we have not gained very much by this equivalent statement. However, other structures might be more useful, as described in \cite{BS15}. A family $\cG \sbse \cP(n)$ is called \emph{simply rooted} if for every $\es \neq G \in \cG$ there is an $x \in G$ with
\begin{align} \label{eq:Interval}
[x, G] \coloneqq \{X \sbse G: x \in X\} \sbse \cG.
\end{align}
It is straightforward to prove (see for example \cite{BBE13} Lemma 18) that $\cF \sbse \cP(n)$ is union closed if and only if $\cG \coloneqq \cP(n) \sm \cF$ is simply rooted. In a similar way to intersection closed sets above we also get an analogue of the Union Closed Sets Conjecture \ref{conj:UCS} for simply rooted families.

More akin to how we will use it below there is also an equivalent way to state the Union Closed Sets Conjecture \ref{conj:UCS} in the languages of lattices and even graphs (see \cite{BS15} for more details and references). In what follows, we will use some theory from \cite{CM03, Day92} and apply it to study the Union Closed Sets Conjecture \ref{conj:UCS}. We will adapt the terminology from \cite{CM03}. Let us start with the following consideration.

Let $\cF \sbse \cP(n)$ be a union closed family with $\es \in \cF$. For every set $X \sbse [n]$
\begin{align} \label{eq:InteriorOperator}
\tau(X) \coloneqq \bigcup \{F \in \cF: F \sbse X\} \in \cF
\end{align}
is the unique maximal set from $\cF$ contained in $X$. This gives a map
$$
\tau: \cP(n) \ra \cF
$$
with the properties
\begin{itemize}
\item [(i)] for all $X \in \cP$ it holds $\tau(X) \sbse X$ (\emph{exclusivity});
\item [(ii)] for all $X \sbse Y \sbse [n]$ it holds $\tau(X) \sbse \tau(Y)$ (\emph{monotonicity});
\item [(iii)] for all $X \sbse [n]$ it holds $\tau(\tau(X)) = \tau(X)$ (\emph{idempotence}).
\end{itemize}
A map $\tau: \cP(n) \ra \cP(n)$ fulfilling the above conditions (i - iii) is called an \emph{interior operator} (notice that \cite{CM03} uses the dual concept of a \emph{closure operator} as is more common in the order theoretic literature, but for union closed families it is more convenient to work with interior operators; all statements below will be taken from \cite{CM03} with according adjustments). It is not hard to show that for a given union closed $\es \in \cF \sbse \cP(n)$ the map $\tau$ from \eqref{eq:InteriorOperator} is an interior operator and that furthermore
$$
\Fix \tau \coloneqq \{X \sbse [n]: \tau(X) = X\} = \cF.
$$
In this way, union closed families containing the empty set are cryptomorphic to interior operators.

\begin{thm} \label{thm:UC<->IO}
Let $n \in \bbN$, the correspondence
\begin{align*}
\{\cF \sbse \cP(n): \es \in \cF \text{ union closed}\} & \ra \{\tau: \cP(n) \ra \cP(n) \text{ interior operator}\} \\
\cF & \mapsto \left(X \mapsto \bigcup\{F \in \cF: F \sbse X\}\right)
\end{align*}
is a bijection with inverse given by
$$
\tau \mapsto \Fix \tau.
$$
\end{thm}
\begin{proof}
See \cite{CM03} Section 2.2 and references therein.
\end{proof}

There does not seem to be a useful way to translate the Union Closed Sets Conjecture \ref{conj:UCS} into the language of interior operators. However, they will turn out useful in the study of the structure of union closed families. To go further into this direction we will continue to study interior operators. Let $\tau: \cP(n) \ra \cP(n)$ be an interior operator and set $\cF \coloneqq \Fix \tau$. For every $F \in \cF$ define
$$
\cT(F) \coloneqq \tau^{-1}(F) \sbse \cP(n).
$$
Then $\bP \coloneqq \{\cT(F): F \in \cF\}$ is a partition of $\cP(n)$ into $\#\cF$ classes. We call a partitioning of $\cP(n)$ that is obtained from an interior operator in this way a \emph{congruence partition}. Every partition implies a corresponding equivalence relation $\gamma$ given by
$$
X \gamma Y \quad \Lra \quad X, Y \in \cT(F) \text{ for some } F \in \cF.
$$
By the way $\bP$ is constructed out of $\tau$ we have
$$
X \gamma Y \quad \Lra \quad \tau(X) = \tau(Y).
$$
An equivalence relation $\gamma$ that is constructed out of an interior operator $\tau$ in the above manner will be called a \emph{congruence relation}. There is an intrinsic way to characterize congruence operator without mentioning any interior operators.

\begin{thm} \label{thm:CharacterizingCR}
Let $\gamma$ be an equivalence relation on $\cP(n)$. Then $\gamma$ is a congruence relation if and only if for all $A, B, C \sbse [n]$ it holds
\begin{align} \label{eq:CR}
A \gamma B \quad \Ra \quad (A \cap C) \gamma (B \cap C).
\end{align}
\end{thm}
\begin{proof}
See \cite{CM03} Section 2.2 and references therein. Notice that congruence relations there are defined in a dual way with unions instead of intersections. This is again due to the fact that we are interested in union closed families and not in intersection closed families.
\end{proof}

Notice that the above theorem gives another cryptomorphic way to study union closed families now via equivalence relations on $\cP(n)$ (and their implied partitions) fulfilling \eqref{eq:CR}. We will continue to study the structure of congruence partitions in further detail.

\begin{cor} \label{cor:CongruencePartition}
Let $\bP = \{P_1, ..., P_m\}$ be a congruence partition. Then every $P_i$ is intersection closed and for all $A, B \in P_i, A \sbse B$ it holds
$$
[A, B] \coloneqq \{X \sbse B: A \sbse X\} \sbse P_i.
$$
\end{cor}
\begin{proof}
Let $P_i$ be an equivalence class of a congruence partition and $\gamma$ the corresponding congruence relation. For the first part, notice that since $A, B \in P_i$, applying \eqref{eq:CR} with $C=A$ we get
$$
A \gamma (A \cap B),
$$
so that (since $A \in P_i$) $A \cap B \in P_i$. For the second part let $A \sbse B$ in $P_i$ and take $A \sbse X \sbse B$. Applying \eqref{eq:CR} with $C=X$ we get
$$
A = (A \cap X) \gamma (B \cap X) = X,
$$
so (since $A \in P_i$) $X \in P_i$.
\end{proof}

\begin{rem} \label{rem:ClassesOfCP}
By the above corollary every class $P_i$ of a congruence partition has a minimal element $X = \bigcap P_i$ and it holds that
$$
P_i = \bigcup_{Y \in P_i} [X, Y].
$$
If $\tau$ is the interior operator associated to the congruence partition, then $X = \tau(Y)$ for all $Y \in P_i$. In particular, the minimal elements of all equivalence classes are precisely the union closed family which generated the interior operator $\tau$ according to \eqref{eq:InteriorOperator}.
\end{rem}

For our purpose the following will be of importance.

\begin{lem} \label{lem:CongruenceEmbedding}
Let $\es \in \cF \sbse \cP(n)$ be a union closed family, $\tau: \cP(n) \ra \cP(n)$ the corresponding interior operator and $\bP = \{\cT(F): F \in \cF\}$ the corresponding congruence partition. Let $E, F \in \cF$ with $E \sbse F$, then $\#\cT(F) \leq \#\cT(E)$.
\end{lem}
\begin{proof}
We will prove that
$$
\iota = \iota_E^F: \cT(F) \ra \cT(E), X \mapsto X \sm (F \sm E)
$$
is an order embedding (that is $X \sbse Y$ if and only if $\iota(X) \sbse \iota(Y)$), in particular injective (see \cite{DP02}). From this the statement follows. \\
We should first show that $\iota$ is well defined, that is if $X \in \cT(F)$ then $X \sm (F \sm E) \in \cT(E)$. Indeed, apply \eqref{eq:CR} with $A = X$, $B = F$ and $C = [n] \sm (F \sm E)$. Since $X, F \in \cT(F)$ we thus get
$$
X \sm (F \sm E) \gamma F \sm (F \sm E) = E \in \cT(E),
$$
so that $X \sm (F \sm E) \in \cT(E)$ as desired. To check that $\iota$ is an order embedding simply note that for all $X, Y \in \cT(F)$ it holds (since $F \sm E \sbse X, Y$)
$$
X \sbse Y \quad \Lra \quad X \sm (F \sm E) \sbse Y \sm (F \sm E).
$$
\end{proof}

In particular we conclude the following.

\begin{cor} \label{cor:Ordering}
Let $\es \in \cF \sbse \cP(n)$ be a union closed family and let $\{\cT(F): F \in \cF\}$ be the corresponding congruence partition. Then there is a way to label the sets from $\cF = \{F_1, \dots, F_m\}$ in such a way that
\begin{itemize}
\item [(i)] $\#\cT(F_1) \leq \#\cT(F_2) \leq ... \leq \#\cT(F_m)$ and
\item [(ii)] if $F_i \spse F_j$ then $i \leq j$.
\end{itemize}
\end{cor}
\begin{proof}
Define the ordering $F_1, \dots, F_m$ by setting $F_1$ to be the unique (by union closedness) maximal set of $\cF$ and, having $F_1, \dots, F_i$ already defined, choosing $F_{i+1} = F$ to be a maximal set from $F \in \cF \sm \{F_1, \dots, F_i\}$ such that $\#\cT(F)$ is of minimal cardinality. The claimed properties (i) and (ii) then follow from Lemma \ref{lem:CongruenceEmbedding} and the given construction respectively.
\end{proof}

\subsection{Up-sets}

An \emph{up-set} (also called \emph{increasing family}) is a family $\cU \sbse \cP(n)$ such that if $A \sbse B \sbse [n]$ with $A \in \cU$, then also $B \in \cU$. That is, up-sets are closed under taking super sets. It is not new to study union closed sets using up-sets via \emph{up-compression techniques} (see \cite{BS15} for details). We will give some insight into this technique, also to demonstrate why the considerations below use up-sets in a nouvelle way. Clearly, every up-set is union closed and it is not hard to show that the Union Closed Sets Conjecture \ref{conj:UCS} holds for all up-sets. By trying to add elements to the sets from a given nontrivial, union closed family one tries to construct an up-set that is easier to work with than general union closed families, but still gives information about the original union closed family.

For what follows, up-sets will play a different roll. Let $\cF \sbse \cP(n)$ be a union closed family and notice that for every $x \in [n]$ we can write
$$
\{F \in \cF: x \in F\} = \cF \cap [x, [n]]
$$
(see \eqref{eq:Interval}). Notice that $[x, [n]]$ is an up-set of cardinality $2^{n-1}$. Thus the following, which we will prove next using the theory developed in the previous section, is a weakening of the Union Closed Sets Conjecture \ref{conj:UCS}.

\begin{thm} \label{thm:WeakUCS}
Let $\cF \sbse \cP(n)$ be a nontrivial, union closed family. There is an up-set $\cU \sbse \cP(n)$ with $\#\cU \leq 2^{n-1}$ such that
$$
\#(\cF \cap \cU) \geq \frac{1}{2} \cdot \#\cU.
$$
\end{thm}
\begin{rem}
It should be noted that the $n$ up-sets of the form $[x, [n]]$ only make up a vanishingly small part out of all up-sets in $\cP(n)$ of cardinality at most $2^{n-1}$. We therefore hope that these ideas can be improved upon to get better results regarding the Union Closed Sets Conjecture \ref{conj:UCS} in the future. We will also demonstrate how the above result can be used to get a statement about the frequency of the most common element among $\cF$ in the next section.
\end{rem}

For the proof we need an elementary lemma.

\begin{lem} \label{lem:Sum}
Let $0 \leq n_1 \leq \dots \leq n_m$ be real numbers, set $N = \sum_{i=1}^m n_i$ and let $\vartheta \in [0, 1]$. Then
$$
\vartheta N \geq \sum_{i=1}^{\lfloor \vartheta m\rfloor} n_i.
$$
\end{lem}

Here $\lfloor \cdot \rfloor$ denotes the \emph{floor-function}. In the proof we will also use the \emph{ceiling-function} $\lceil \cdot \rceil$.

\begin{proof}
The statement is clear for $\vartheta < \frac{1}{m}$, since then the right hand side is zero. Suppose $\vartheta \geq \frac{1}{m}$, define the function $f: (0, m] \ra \bbR, f(x) \coloneqq n_{\lceil x \rceil}^{}$. By monotonicity for the $n_i$'s the function $f$ is also monotonically increasing. Thus, since $\frac{m}{\lfloor \vartheta m \rfloor} \geq 1$, we have
$$
f(x) \leq f \left(\frac{m}{\lfloor \vartheta m \rfloor} \cdot x\right)
$$
for $x \in (0, \lfloor \vartheta m \rfloor]$. By the substitution $y = \frac{m}{\lfloor \vartheta m \rfloor} \cdot x$ we get
\begin{align*}
\sum_{i=1}^{\lfloor \vartheta m\rfloor} n_i & = \int_0^{\lfloor \vartheta m\rfloor} f(x) dx \leq \int_0^{\lfloor \vartheta m\rfloor} f \left(\frac{m}{\lfloor \vartheta m \rfloor} \cdot x\right) dx = \int_0^m f(y) \cdot \frac{\lfloor \vartheta m \rfloor}{m} dy \\
 & = \frac{\lfloor \vartheta m \rfloor}{m} \cdot \sum_{i=1}^m n_i \leq \vartheta N.
\end{align*}
\end{proof}

We can now proof Theorem \ref{thm:WeakUCS} in a more general version.

\begin{thm} \label{thm:UpsetIntersections}
Let $n, t \in \bbN$ and let $\cF \sbse \cP(n)$ be a nontrivial, union closed. Then there is an up-set $\cU \sbse \cP(n)$ with $\#\cU \leq \lceil 2^n / t \rceil$ and
$$
\#(\cF \cap \cU) \geq \frac{1}{t} \cdot \#\cF.
$$
\end{thm}
\begin{proof}
We may assume $\es \in \cF$. Order $\#\cF = \{F_1, \dots, F_m\}$ as in Corollary \ref{cor:Ordering}. We claim that $\cU \coloneqq \bigcup_{i=1}^{\lceil m/t \rceil} \cT(F_i)$ does the job. First, we have $\cF \cap \cU = \{F_1, \dots, F_{\lceil m/t \rceil}\}$ so $\#(\cF \cap \cU) \geq \frac{1}{t} \cdot \#\cF$. It remains to show that $\cU$ is an up-set with
$$
\#\cU \leq \left\lceil \frac{1}{t} \cdot 2^n \right\rceil.
$$
The fact that $\cU$ is an up-set stems from the fact $\{\cT(F_1), ..., \cT(F_{\lceil m/t \rceil})\}$ is a partition of $\cU$ fulfilling property (i) from Corollary \ref{cor:Ordering}. To bound $\#\cU$ notice that $F_1 = [n]$ by nontriviality of $\cF$. Setting $f_i = \#\cT(F_i)$ we thus have $1 = f_1 \leq f_2 \leq \dots \leq f_m$ and
$$
\sum_{i=2}^m f_i = 2^n-1.
$$
Applying Lemma \ref{lem:Sum} to $n_1 = f_2, \dots, n_{m-1} = f_m$ with $\vartheta = \frac{\lceil\frac{m}{t}\rceil - 1}{m-1}$ we get
\begin{align*}
\#\cU & = \sum_{i=1}^{\lceil m/t \rceil} f_i = 1 + \sum_{i=1}^{\lceil m/t \rceil-1} n_i \leq 1 + \frac{\lceil\frac{m}{t}\rceil - 1}{m-1} \cdot \sum_{i=1}^{m-1} n_i \\
 & \leq 1+ \frac{\frac{m+t-1}{t}-1}{m-1} (2^n-1) = 1 + \frac{1}{t} \cdot (2^n-1) < \frac{1}{t} \cdot 2^n + 1,
\end{align*}
but since $\#\cU$ is an integer we get the desired bound.
\end{proof}

Setting $t=2$ yields Theorem \ref{thm:WeakUCS}.

\subsection{A Demonstration}

We demonstrate how one can apply Theorem \ref{thm:WeakUCS} to get a bound about the frequency of the most frequent element among a nontrivial, union closed family $\cF$. The here proven bound is unfortunately worse than all already known bounds which we will collect further below. The author hopes that these results still turn out fruitful in the further development of the Union Closed Sets Conjecture \ref{conj:UCS}. The following result might be interesting on its own.

\begin{thm} \label{thm:IntersectingInUpset}
Let $n \in \bbN, n \geq 2$ and $\cU \sbse \cP(n)$ be an upset with $\#\cU \leq 2^{n-1}$. Let $C = 1+e^{-1} = 1.367\dots$ and $t \in \bbN, t \geq \frac{Cn}{\log_2 n}$. For any collection of sets $A_1, \dots, A_t \in \cU$ there are indices $i < j$ with $A_i \cap A_j \neq \es$.
\end{thm}

That is, for any sufficiently large collection of sets from a sufficiently small up-set there are two intersecting sets from that collection.

\begin{proof}
Assume $\cU \sbse \cP(n)$ is an up-set with $\#\cU \leq 2^{n-1}$ and $A_1, \dots, A_t \in \cU$ are pairwise disjoint. We aim to bound $t$. For all $i=1, \dots, t$ then $[A_i, [n]] \sbse \cU$, so
$$
\bigcup_{i=1}^t [A_i, [n]] \sbse \cU.
$$
It holds that $\bigcap_{i \in I} [A_i, [n]] = [\bigcup_{i \in I} A_i, [n]]$ for all nonempty set of indices $I \sbse [t]$. For $i \in [t]$ set $a_i \coloneqq \#A_i$. Since the $A_i$'s are pairwise disjoint it holds
$$
\#\bigcup_{i \in I} A_i = \sum_{i \in I} a_i
$$
for all nonempty $I \sbse [t]$. By inclusion-exclusion we get
\begin{align*}
\#\cU & \geq \# \bigcup_{i=1}^t [A_i, [n]] = \sum_{\es \neq I \sbse [t]} (-1)^{\#I - 1} \cdot \# \bigcap_{i \in I} [A_i, [n]] = \sum_{\es \neq I \sbse [t]} (-1)^{\#I - 1} \cdot 2^{n - \sum_{i \in I} a_i} \\
 & = 2^n \left( 1-\prod_{i=1}^t (1-2^{-a_i})\right).
\end{align*}
For variables $x_1, \dots, x_t \geq 0$ under the constraint $\sum_{i=1}^t x_i \leq n$, the expression
$$
1-\prod_{i=1}^t (1-2^{-x_i})
$$
is minimized for $x_1 = \dots = x_t = \frac{n}{t}$. This yields
$$
2^{n-1} \geq \#\cU \geq 2^n \left(1-\left(1-2^{-n/t}\right)^t\right) = 2^n - \left(2^{n/t}-1\right)^t,
$$
equivalently
$$
n \geq t \cdot \log_2 \left( \frac{1}{1-2^{-1/t}}\right).
$$
Using $e^{-x} > 1-x$ with $x = \frac{\ln 2}{t}$ one obtains
$$
n > t \cdot \log_2\left(\frac{t}{\ln 2}\right).
$$
Straightforward but lengthy calculations finally show
$$
t < \frac{Cn}{\log_2 n}.
$$
\end{proof}

\begin{rem}
In the above theorem, it seems to be possible to relax the condition
$$
t \geq \frac{C n}{\log_2 n}
$$
to
$$
t \geq \frac{(1+o(1))n}{\log_2 n},
$$
where $o(1)$ denotes a function varying in $n$ and tending to $0$ as $n$ tends to infinity. While this will not affect the considerations below very much, it might be of separate interest to investigate optimal conditions for Theorem \ref{thm:IntersectingInUpset} and related statements.
\end{rem}

To continue we repeat some notions from graph theory, see \cite{Die16} for a detailed introduction. For a set $X$ let ${X \choose 2}$ be the set of two element subsets of $X$. Let $G = (V, E)$ be a (simple) graph. A \emph{clique} in $G$ is a subset $X \sbse V$ with
$$
{X \choose 2} \sbse E.
$$
The \emph{clique number} $\omega(G)$ of $G$ is the size of a largest clique in $G$. A set $A \sbse V$ is an \emph{independent set} in $G$ if
$$
{A \choose 2} \cap E = \es.
$$
The \emph{independence number} $\alpha(G)$ of $G$ is the size of a largest independent set in $G$. The inequality $\alpha(G) < t$ is equivalent to the statement that any $t$ vertices contain at least one edge. Denoting by $\overline{G} \coloneqq (V, {V \choose 2} \sm E)$ the \emph{complement} of $G$, by definition it holds $\alpha(G) = \omega(\overline{G})$. \emph{Turán's theorem} \cite{Tur41} gives a bound on the number of edges in a graph with a given clique number.

\begin{thm} \label{thm:Turan}
Let $G = (V, E)$ be a graph on $\#V = n$ vertices and $\#E = m$ edges. Let $t \in \bbN$.
\begin{itemize}
\item [(i)] If $\omega(G) < t$ then $m \leq \left(1-\frac{1}{t-1}\right) \cdot \frac{n^2}{2}$.

\item [(ii)] If $\alpha(G) < t$ then $m \geq \frac{1}{t-1} \cdot \frac{n^2}{2} - \frac{n}{2}$.
\end{itemize}
\end{thm}
\begin{proof}
For (i) see \cite{Die16}. For (ii) apply (i) to the complement $\overline{G}$.
\end{proof}

We will use Turán's theorem together with Theorem \ref{thm:IntersectingInUpset} to get a bound on the frequency of the most frequent element among a nontrivial, union closed family.

\begin{thm} \label{thm:FreqencyBound}
Let $n \in \bbN, n \geq 2$ and let $\cF \sbse \cP(n)$ be a nontrivial, union closed family. Set $\#\cF = m$ and assume
$$
m \geq \frac{7n}{\log_2 n}.
$$
Then there is an $x \in [n]$ with
$$
\#\{F \in \cF: x \in F\} \geq \frac{\sqrt{\log_2 n}}{3n} \cdot \#\cF.
$$
\end{thm}
\begin{proof}
By Theorems \ref{thm:WeakUCS} and \ref{thm:IntersectingInUpset} there is a subfamily $\cE \sbse \cF$ with $\mu := \#\cE \geq \frac{m}{2}$ and such that, setting $t := \left\lceil \frac{Cn}{\log_2 n}\right\rceil$ with $C = 1+e^{-1}$, for any $A_1, ..., A_t \in \cE$ at least two of these sets have a nonempty intersection. Consider the graph $G = (\cE, E(G))$ with
$$
E(G) = \left\{E_1E_2 \in {\cE \choose 2}: E_1 \cap E_2 \neq \es\right\}.
$$
By construction $\alpha(G) < t$, so that by Theorem \ref{thm:Turan} (ii) we have
$$
\#E(G) \geq \frac{1}{t-1} \cdot \frac{\mu^2}{2} - \frac{\mu}{2}.
$$
For every $E_1E_2 \in E(G)$ pick a $c(E_1E2) \in E_1 \cap E_2$. This defines an edge coloring $c: E(G) \ra [n]$. Consequently, there is an $x \in [n]$ that appears on at least
\begin{align} \label{ineq:Colors}
\frac{1}{n} \cdot \left( \frac{1}{t-1} \cdot \frac{\mu^2}{2} - \frac{\mu}{2} \right)
\end{align}
edges. Let $G'$ be the graph induced by the edges of color $x$ and let $\mu'$ be the number of vertices in $G'$, so that (using \eqref{ineq:Colors} and $t \leq \frac{Cn}{\log_2 n} + 1$)
$$
\frac{(\mu')^2}{2} \geq {\mu' \choose 2} \geq \frac{1}{n} \cdot \left( \frac{1}{t-1} \cdot \frac{\mu^2}{2} - \frac{\mu}{2} \right) \geq \frac{\mu^2}{2n \cdot \frac{Cn}{\log_2 n}} - \frac{\mu}{2n}, 
$$
so that
$$
\mu' \geq \sqrt{\frac{\mu^2 \log_2 n}{Cn^2} - \frac{\mu}{n}}
$$
The right hand side for $\mu \geq \frac{m}{2}$ is minimized at $\mu = \frac{m}{2}$ (using the assumptions from the statement), so that
$$
\mu' \geq \sqrt{\frac{m^2 \log_2 n}{4Cn^2} - \frac{m}{2n}} = \sqrt{\frac{1}{4C} - \frac{n}{2m \log_2 n}} \cdot \frac{\sqrt{\log_2 n}}{n} \cdot m.
$$
Using again the assumptions we finally obtain
$$
\#\{F \in \cF: x \in F\} \geq \mu' \geq \sqrt{\frac{1}{4C} - \frac{1}{14}} \cdot \frac{\sqrt{\log_2 n}}{n} \cdot m \geq \frac{\sqrt{\log_2 n}}{3n} \cdot m.
$$
\end{proof}

\begin{rem}
There are already known lower bounds on the frequency of a most frequent element in $\cF$, see \cite{BS15}. In particular, by \cite{Bal11, Rei03, Woj99} respectively (the third being an improvement by a constant of a bound in \cite{Kni94}), it is known that there is an element that is contained in an
$$
\Omega\left(\max\left\{\sqrt{\frac{\log_2 n}{n}}, \frac{\log_2 m}{n}, \frac{1}{\log_2 m}\right\}\right)
$$
-fraction of all sets from $\cF$ (for $n$ and $m$ sufficiently large). The first lower bound supersedes the bound from Theorem \ref{thm:FreqencyBound}. Notice however that the above proof does not use all the information known about $\cE$. In particular, one could assume $\cE$ to be union closed. Also, one might see that the above technique via Turán's theorem might not be optimal since the intersection $E_1 \cap E_2$ might be very large so that there are a lot of possibilities for $c(E_1E_2)$. There is some hope that more refined arguments also yield a better bound.
\end{rem}

\subsection{Intersecting Families}

The idea of the proof of Theorem \ref{thm:FreqencyBound} was that any union closed family $\cF$ contains a large subfamily $\cE$ so that any sufficiently large quantity of sets from $\cE$ must contain two intersecting sets. One way one could improve the bound from Theorem \ref{thm:FreqencyBound} is via the following open question. A family of sets $\cE$ is called \emph{intersecting} if for all $A, B \in \cE$ it holds $A \cap B \neq \es$.

\begin{q} \label{que:Intersecting}
For every nontrivial, union closed family $\cF \sbse \cP(n)$ does there exist an intersecting subfamily $\cE \sbse \cF$ with $\#\cE \geq \frac{1}{2} \cdot \#\cF$?
\end{q}

This is again a weaker version of the Union Closed Sets Conjecture \ref{conj:UCS}, but would strengthen Theorem \ref{thm:WeakUCS}. In this context, it is natural to ask about the frequency of a most frequent element in an intersecting family. Even though intersecting families are well studied objects in combinatorics (dating back to \cite{EKR61} for example), this aspect does not seem to have been investigated so far.

\begin{thm} \label{thm:FrequencyIntersecting}
Let $\cE \sbse \cP(n)$ be an intersecting family of size $\#\cE = m$. There is an $x \in [n]$ with
$$
\#\{E \in \cE: x \in F\} \geq \frac{1}{2} + \sqrt{\frac{1}{4} + \frac{m^2-m}{n}} \geq \sqrt{\frac{m-1}{mn}} \cdot \#\cE.
$$
\end{thm}

The proof is an adapted version of the proof of Theorem \ref{thm:FreqencyBound}.

\begin{proof}
Consider the graph $G = (\cE, {\cE \choose 2})$ and a coloring
$$
c: {\cE \choose 2} \ra [n], c(E_1E_2) \in E_1 \cap E_2.
$$
There is a color $x \in [n]$ that appears on at least
$$
\frac{1}{n} \cdot {m \choose 2}
$$
edges from $G$. Let $G'$ be the subgraph induced by the edges of color $x$ and let $m'$ be the number of vertices in $G'$. Thus
$$
{m' \choose 2} \geq \frac{1}{n} \cdot {m \choose 2},
$$
equivalently
$$
m' \geq \frac{1}{2} + \sqrt{\frac{1}{4} + \frac{m^2-m}{n}}.
$$
Since (by construction) every vertex from $G'$ is a set containing $x$, the claim follows.
\end{proof}

The bound from the above theorem can be sharp, for example for \emph{projective planes} (see \cite{BJL99} for details). For $\#\cE \geq 2$ the theorem gives an element contained in at least $(2n)^{-1/2} \cdot \#\cE$. Together with Question \ref{que:Intersecting} one would then get an element contained in at least an $\Omega(n^{-1/2})$-fraction of sets from a nontrivial, union closed family $\cF$. While this is again worse than the bound from \cite{Bal11}, one could again try to refine the argument from above. In particular, one can consider the following question.

\begin{q} \label{que:UCIntersecting}
What can be said about the frequency of the most frequent element in a nontrivial, union closed, intersecting family $\cF \sbse \cP(n)$?
\end{q}

We finish by stating that it can also be of interest to combine Question \ref{que:UCIntersecting} with Question \ref{que:UCSFrequencies}. We leave this for future research.

\section{Acknowledgement}

I would like to thank my mentors Christoph Helmberg, Martin Winter and Tino Ullrich of the TU Chemnitz and Tibor Szabó of the FU Berlin, as well as all my colleagues of the TU Chemnitz and FU Berlin for their continued support.


\begin{thebibliography}{9999}

\bibitem{AEL20}
\textsc{J. Aaronson, D. Ellis, I. Leader}, A Note on Transitive Union-Closed Families, \textit{The Electronic Journal of Combinatorics} \textbf{28(2)} \#P2.3 (2021)

\bibitem{ADPRT20}
\textsc{B. Amaral, L. Dalton, D. Polakowski, A. Raymond, B. Thomas}, The Linear Relaxation of an Integer Program for the Union-Closed Conjecture, \href{https://arxiv.org/abs/2004.05210}{\textit{arXiv:2004.05210}} (2020)

\bibitem{Bal11}
\textsc{I. Balla}, Minimum density of union-closed families, \href{https://arxiv.org/abs/1106.0369}{\textit{arXiv:1106.0369}} (2011)

\bibitem{BBE13}
\textsc{I. Balla, B. Bollobás, T. Eccles}, Union-closed families of sets, \textit{Journal of Combinatorial Theory, Series A} \textbf{120}, Issue 3, pp. 531-544 (2013)

\bibitem{BJL99}
\textsc{T. Beth, D. Jungnickel, H. Lenz}, Design Theory, \textit{Cambridge University Press}, Volume 1, Second Edition (1999)

\bibitem{BD18}
\textsc{G. Brinkmann, R. Deklerck}, Generation of Union-Closed Sets and Moore Families, \textit{Journal of Integer Sequences} \textbf{21}, Article 18.1.7 (2018)

\bibitem{BS15}
\textsc{H. Bruhn, O. Schaudt}, The Journey of the Union-Closed Sets Conjecture, \textit{Graphs and Combinatorics} \textbf{31}, pp. 2043-2074 (2015)

\bibitem{CM03}
\textsc{N. Caspard, B. Monjardet}, The lattices of closure systems, closure operators, and implicational systems on a finite set: a survey, \textit{Discrete Applied Mathematics} \textbf{127}, Issue 2, pp. 241-269 (2003)

\bibitem{DP02}
\textsc{B. A. Davey, H. A. Priestley}, Introduction to Lattices and Order, \textit{Cambridge University Press}, Second Edition (2002)

\bibitem{Day92}
\textsc{A. Day}, The Lattice Theory of Functional Dependencies and Normal Decompositions, \textit{International Journal of Algebra and Computation} \textbf{2}, Number 4, pp. 409-431 (1992)

\bibitem{Die16}
\textsc{R. Diestel}, Graph Theory, \textit{Springer-Verlag, Graduate Texts in Mathematics}, Volume 173, Fifth Edition (2016/17)

\bibitem{EIL22}
\textsc{D. Ellis, M.-R. Ivan, I. Leader}, Small sets in union-closed families, \href{https://arxiv.org/abs/2201.11484}{\textit{arXiv:2201.11484}} (2022)

\bibitem{EKR61}
\textsc{P. Erdős, C. Ko, R. Rado}, Intersection Theorems for Systems of Finite Sets, \textit{The Quarterly Journal of Mathematics}, Volume 12, Issue 1, pp. 313-320 (1961)

\bibitem{Gow16}
\textsc{W. T. Gowers, et al.}, \href{https://gowers.wordpress.com/2016/01/21/frankls-union-closed-conjecture-a-possible-polymath-project/}{\textit{gowers.wordpress.com/2016/01/21/frankls-union-closed-conjecture-a-possible-polymath-project/}} (2016)

\bibitem{Kar17}
\textsc{I. Karpas}, Two Results on Union-Closed Families, \href{https://arxiv.org/abs/1708.01434}{\textit{arXiv:1708.01434}} (2017)

\bibitem{Kni94}
\textsc{E. Knill}, Graph Generated Union-closed Families of Sets, \href{https://arxiv.org/abs/math/9409215}{\textit{arXiv:math/9409215}} (1994)

\bibitem{LRS12}
\textsc{U. Leck, I. T. Roberts, J. Simpson}, Minimizing the weight of the union-closure
of families of two-sets, \textit{Australasian Journal of Combinatorics} \textbf{52}, pp. 67-73 (2012)

\bibitem{Mas16}
\textsc{J. Maßberg}, The Union-Closed Sets Conjecture for Small Families, \textit{Graphs and Combinatorics} \textbf{32}, pp. 2047-2051 (2016)

\bibitem{PSZ12}
\textsc{Y. Peng, P. Sissokho, C. Zhao}, An extremal problem for set families generated with the union and symmetric difference operations, \textit{Journal of Combinatorics} \textbf{3}, Number 4, pp. 651-668 (2012)

\bibitem{Raz17}
\textsc{A. Raz}, Note on the union-closed sets conjecture, \textit{The Electronic Journal of Combinatorics} \textbf{24(3)} \#P3.53 (2017)

\bibitem{Rei03}
\textsc{D. Reimer}, An Average Set Size Theorem, \textit{Combinatorics, Probability and Computing} \textbf{12}, Issue 1, pp. 89-93 (2003)

\bibitem{Stu20}
\textsc{L. Studer}, An asymptotic version of the union-closed sets conjecture, \textit{The American Mathematical Monthly}, Volume 128, Issue 7, pp. 652-654 (2021)

\bibitem{Tia21}
\textsc{C. Tian}, Union-closed Sets Conjecture Holds for Height $H(\cF) \leq 3$
and $H(\cF) \geq n-1$, \href{https://arxiv.org/abs/2112.06659}{\textit{arXiv:2112.06659}} (2021)

\bibitem{Tur41}
\textsc{P. Turán}, On an extremal problem in graph theory, \textit{Matematikai és Fizikai Lapok} \textbf{48}, pp. 436-452 (1941)

\bibitem{VZ17}
\textsc{B. Vučković, M. Živković}, The 12-Element Case of Frankl’s Conjecture, \textit{IPSI BgD Transactions on Internet Research}, Volume 13, Number 1, pp. 65-71 (2017)

\bibitem{Woj99}
\textsc{P. W\'{o}jcik}, Union-closed families of sets, \textit{Discrete Mathematics} \textbf{199}, Issues 1-3, pp. 173-182 (1999)

\end{thebibliography}
\end{document}